\documentclass[12pt]{amsart}
\usepackage{amsmath, amsfonts, amsthm}
\usepackage{stmaryrd}
\usepackage{quiver}
\usepackage[a4paper, margin=1in]{geometry}
\usepackage{tikz-cd}
\usepackage{enumitem}
\usepackage{mathrsfs}
\usepackage{amsfonts, amsmath, amssymb, amsfonts}
\usepackage{amsthm}
\usepackage{hyperref}
\usepackage[dvipsnames]{xcolor}
\usepackage{mathtools}
\hypersetup{
    colorlinks=true,
    linkcolor=blue,
    filecolor=magenta,      
    urlcolor=cyan,
    citecolor=red
}
\usepackage{cleveref}
\usepackage{hyperref}
\hypersetup{
    colorlinks=true,
    linkcolor=blue,
    filecolor=magenta,      
    urlcolor=cyan,
    citecolor=red
}
\newcommand{\diffstrat}{\Mcal_{5, 4}^{(2)}}

\newtheorem{theorem}{Theorem}[section]
\newtheorem{lemma}[theorem]{Lemma}
\newtheorem{proposition}[theorem]{Proposition}
\newtheorem{corollary}[theorem]{Corollary}

\theoremstyle{remark}

\theoremstyle{definition}
\newtheorem{definition}[theorem]{Definition}

\newcommand{\ev}{\mathrm{ev}}

\newcommand{\Ebf}{\mathbf{E}}
\newcommand{\Fbf}{\mathbf{F}}

\newcommand{\Bscr}{\mathscr{B}}
\newcommand{\Qscr}{\mathscr{Q}}
\newcommand{\Yscr}{\mathscr{Y}}

\newcommand{\Xscr}{\mathscr{X}}

\newcommand{\Fscr}{\mathscr{F}}

\newcommand{\Ascr}{\mathscr{A}}

\newcommand{\Bcal}{\mathcal{B}}
\newcommand{\Acal}{\mathcal{A}}
\newcommand{\Hcal}{\mathcal{H}}
\newcommand{\Rcal}{\mathcal{R}}
\newcommand{\Ccal}{\mathcal{C}}
\newcommand{\Wcal}{\mathcal{W}}

\newcommand{\Scal}{\mathcal{S}}

\newcommand{\Ical}{\mathcal{I}}
\newcommand{\Vcal}{\mathcal{V}}
\newcommand{\Ocal}{\mathcal{O}}
\newcommand{\Ucal}{\mathcal{U}}

\newcommand{\Tcal}{\mathcal{T}}
\newcommand{\Mcal}{\mathcal{M}}

\newcommand{\PP}{\mathbb{P}}

\newcommand{\CC}{\mathbb{C}}

\newcommand{\QQ}{\mathbb{Q}}

\newcommand{\GG}{\mathbb{G}}

\newcommand{\pbf}{\mathbf{p}}

\newcommand{\rarr}{\rightarrow}

\newcommand{\gitquot}{/ \! \! /}

\newcommand{\mg}{\Mcal_g}

\newcommand{\mgn}{\Mcal_{g, n}}

\newcommand{\sg}{\Scal_g}

\newcommand{\sfive}{\Scal_5^-}
\newcommand{\sfivetilde}{\widetilde\Scal_5^-}

\newcommand{\Sing}{\textup{Sing}}

\newcommand{\pr}{\textup{pr}}
\newcommand{\disc}{\textup{disc}}
\newcommand{\mult}{\textup{mult}}

\newcommand{\SL}{\textup{SL}}

\newcommand{\PGL}{\textup{PGL}}
\newcommand{\GL}{\textup{GL}}
\renewcommand{\H}{\textup{H}}
\renewcommand{\P}{\textup{P}}
\newcommand{\R}{\textup{R}}

\newcommand{\Ker}{\textup{Ker}}
\renewcommand{\Im}{\textup{Im}}

\newcommand{\Pic}{\textup{Pic}}

\newcommand{\Aut}{\textup{Aut}}

\newcommand{\supp}{\textup{supp}}
\newcommand{\Bs}{\textup{Bs}}

\newcommand{\CH}{\textup{CH}}
\newcommand{\Sym}{\textup{Sym}}
\newcommand{\cl}{\textup{cl}}

\newcommand{\Sfrak}{\mathfrak{S}}

\begin{document}
\author{Bogdan Carasca}
\address{Humboldt Universität zu Berlin, Unter den Linden 6, 10117 Berlin}
\email{bogdanpetru.carasca@gmail.com}

\begin{abstract}
    The moduli spaces $\sg^-$ parametrise odd spin curves of genus $g$. These are pairs $[C, \eta]$ where $C$ is a smooth genus $g$ curve of and $\eta$ is a line bundle on $C$ such that $\eta^{\otimes 2} = \omega_C$ and $h^0(C, \eta)$ is odd. The main result of this work is the tautology of the Chow ring of $\sfive$. Our method of proof revolves around an analysis of the geometry of canonical genus 5 curves and totally tangent hyperplanes. In the course of establishing our main result, we also prove the rationality of the closely related differential stratum in $\Mcal_{5, 4}$ dominating $\sfive$.
\end{abstract}

\title[The Chow ring of $\Scal_5^-$] {The Chow ring of $\sfive$ is tautological}
\maketitle

\let\thefootnote\relax

\section{Introduction}

The moduli spaces $\sg$ parametrise objects of the form $[C, \eta]$ where $C$ is a smooth curve of genus $g$ and $\eta \in \Pic(C)$ is a line bundle such that $\eta^{\otimes 2} = \omega_C$. Such a a pair is called a spin curve. The forgetful map $\pi : \sg \rarr \mg$ exhibits a direct connection between $\sg$ and $\mg$, and it is finite of degree $2^{2g}$. Unlike $\mg$, the moduli spaces $\sg$ are not connected. In fact, $\sg$ has two irreducible components, $\sg^+$ and $\sg^-$, the moduli spaces of even (resp. odd) spin curves of genus $g$ -- see \cite{mumford1971thetaChar} and  \cite{atiyah1971spin}. We say that a spin curve $[C, \eta]$ is \textit{even} (resp. \textit{odd}) if the number of sections $h^0(C, \eta)$ is even (resp. odd).

\vskip 0.5em
In the landmark paper \cite{mumford1983towardEnumerativeGeometryMg}, Mumford began the study of the tautological ring $\R^\bullet(\mg)$ of $\mg$. Recall that $\R^\bullet(\mg)$ is the $\QQ$--subalgebra of $\CH^\bullet(\mg)$ generated by the $\kappa$--classes $\kappa_i := f_*(c_i(\omega_f)^{i +1})$, where $f : \Ccal_g \rarr \mg$ is the universal curve and $\omega_f$ is the relative dualising sheaf. Mumford then went on to ask whether the Chow ring $\CH^\bullet(\mg)$ could be tautological, at least in small genus. Regarding this question, significant progress has been made. With $\QQ$--coefficients, one finds the following results:
\begin{itemize}[label=--]
    \item Mumford determines the Chow ring $\CH^\bullet(\overline\Mcal_2)$ in \cite{mumford1983towardEnumerativeGeometryMg}.
    \item Faber determines the Chow rings $\CH^\bullet(\overline\Mcal_3)$ and $\CH^\bullet(\Mcal_4)$ in \cite{faber199ChowRingM3} and \cite{faber1990chowM4}.
    \item Izadi determines the Chow ring $\CH^\bullet(\Mcal_5)$ in \cite{izadi1995chowM5}.
    \item Vakil and Penev determine $\CH^\bullet(\Mcal_6)$ in \cite{vakil2015chowRingM6}.
    \item Canning and H. Larson determine $\CH^\bullet(\Mcal_g)$ for $g = 7, 8, 9$ in \cite{larson2024chowMg}.
\end{itemize}
In all the works above, the key step is showing the tautology of the Chow ring which, using work by Faber in \cite{faber1999conjDescTautMg}, allows one to compute the ideal of relations.

\vskip 0.5em
In contrast, much less is known about this problem for $\sg^\pm$. Define the tautological ring of $\sg^\pm$ (with rational coefficients) to be the subalgebra of $\CH^\bullet(\sg^\pm)$ generated by pullbacks of tautological classes on $\mg$ under the forgetful morphism $\sg^\pm \rarr \mg$. The cohomological tautological ring of $\sg^\pm$ is the image under the cycle class map $\cl : \CH^\bullet(\sg^\pm) \rarr \H^\bullet(\sg^\pm)$ of $\R^\bullet(\sg^\pm) \subset \CH^\bullet(\sg^\pm)$.
\begin{theorem}\label{thm: sfive has CH = R}
    The moduli space $\Scal_5^-$ of odd spin curves of genus $5$ has tautological Chow and cohomology rings.
\end{theorem}
Notice that Theorem \ref{thm: sfive has CH = R} in particular implies that $\pi^* : \CH^\bullet(\Mcal_5) \rarr \CH^\bullet(\sfive)$ is surjective so that, in fact, $\pi^*$ is an isomorphism. Since the Chow ring of $\Mcal_5$ is known, cf. \cite{izadi1995chowM5}, we obtain the following explicit formula for the Chow ring of $\sfive$.
\begin{corollary}
    The Chow and cohomology rings of $\sfive$ are given by
    \[
    \CH^\bullet(\sfive) \cong \H^\bullet(\sfive) \cong \QQ[\kappa_1]/(\kappa_1^4).
    \]
\end{corollary}

Along the way towards the tautology of $\CH^\bullet(\Scal_5^-)$, we also show a rationality result for the stratum of differentials
\[
\Mcal_{5, 4}^{(2)}:= \{[C, p_1, \dots, p_4] \in \Mcal_{5, 4}\colon \Ocal_C(2p_1 + \dots + 2p_4) \cong \omega_C\},
\]
which is a $\Sfrak_4$--cover of $\Scal_5^-$.
\begin{theorem}\label{thm: rationality of diff stratum}
    The differential stratum $\Mcal_{5, 4}^{(2)}$ is a rational variety.
\end{theorem}

Theorem \ref{thm: rationality of diff stratum} builds on \cite{farkas2014geometry} where Farkas and Verra, using Mukai models, provide (uni)rational parametrisations for moduli spaces of odd spin curves. Moreover, they complete the Kodaira classification of $\sg^-$ and we now know that $\sg^-$ is unirational for $g \le 8$, uniruled for $g \le 11$, and of general type for $g \ge 12$. In our situation, the differential stratum $\Mcal_{5, 4}^{(2)}$ is proven to be rational after a studying the geometry of the restriction of nets of quadrics in $\PP^4$ to theta hyperplanes. This approach is manifestly dependent on the specific geometry of the general canonical curve of genus 5.

\subsection{Conventions and terminology}

Throughout we work over $\CC$. All Chow and cohomology rings are with rational coefficients. We maintain the subspace convention for Grassmannians. We denote by $X/G$ the quotient stack and by $X\gitquot G$ the GIT quotient.

\subsection{Vistoli's theorem}

This subsection contains a reminder of a result by Vistoli which will be essential to our approach to the Chow ring $\CH^\bullet(\sfive)$.
\begin{theorem}[\textup{cf. }{\cite[Theorem 2]{vistoli1987chowGroupsQutients}}]\label{thm: vistoli thm chow groups of quotients}
	Suppose $G = \GL(r)$ or $\SL(r)$ acts on a quasi-projective variety $X$ and let $\pi : X \rarr Y := X/G$ be the quotient. Denote by $\Vcal \rarr Y$ the vector bundle on $Y$ corresponding to the principal $G$--bundle $X \rarr Y$. The following hold:
	\begin{enumerate}[label=$\roman*)$]
		\item The pullback $\pi^* : \CH^\bullet(Y) \rarr \CH^\bullet(X)$ is surjective.
		\item The kernel of $\pi^*$ is generated as a group by classes of the form $c_i(\Vcal) \cap [Z]$ where $i > 0$ and $[Z] \in \CH^\bullet(Y)$.
		\item If $C$ (hence $Y$) are smooth, then $\CH^\bullet(Y)$ is generated as a ring by $c_1(\Vcal), \dots, c_r(\Vcal)$ and any chosen set of lifts of generators for $\CH^\bullet(X)$.
	\end{enumerate}
\end{theorem}

\subsection{Acknowledgements}

The research was partly conducted in the framework of the DFG-funded research training group RTG 2965: From Geometry to Numbers, Project number 512730679.

\vskip 0.5em
I thank my advisor Gabi Farkas for his guidance and support throughout this work. I also thank Marian Aprodu, Riccardo Redigolo, Vlad Robu, C\u{a}lin Spiridon and Andrei Stoenic\u{a} for helpful conversations.

\section{The rationality of \texorpdfstring{$\diffstrat$}{}}\label{section: rationality of diffstrat}

In this section we are going to discuss a birational model for the differential stratum $\diffstrat$ which parametrises smooth 4--pointed genus 5 curves $[C, p_1, \dots, p_4]$ such that $\Ocal_C(2p_1 + \dots + 2p_4) \cong \omega_C$. Using this model, Theorem \ref{thm: rationality of diff stratum} will then be proven. In our situation, this is a natural moduli space to study due to its relation to $\sfive$, namely $\sfive \approx \Mcal_{5,4}^{(2)} / \Sfrak_4$. This birational isomorphism simply maps $[C, p_1 + \dots + p_4] \in \Mcal_{5, 4}^{(2)}/\Sfrak_4$ to $[C, \Ocal_C(p_1 + \dots + p_4)] \in \sfive$.

\vskip 0.5em
We start by motivating our construction. Let $[C, p_1, \dots, p_4] \in \diffstrat$ be non--hyperellipic non--trigonal so that the canonical curve $C \subset \PP^4$ is a complete intersection of 3 quadrics $Q_1, Q_2, Q_3$. The points $p_1, \dots, p_4$ span a hyperplane $H \subset \PP^4$ which is totally tangent to $C$ along the divisor $p_1 + \dots + p_4$. Restricting the net of quadrics $\Lambda := \langle Q_1, Q_2, Q_3\rangle$ to $H \cong \PP^3$, we obtain a net of quadrics in $H$ which we denote by $\Gamma$. The base locus of $\Gamma$ is a zero dimensional subscheme of $\PP^3$ such that $\supp\bigl(\Bs\,\Gamma\bigr) = p_1 + \dots + p_4$ and $\mult_{p_i} \bigl(\Bs\,\Gamma\bigr) = 2$ for $i = 1, \dots, 4$.

\vskip 0.5em
Let $t_0, \dots, t_4$ be homogeneous coordinates on $\PP^4$ and embed $\PP^3$ in $\PP^4$ as the hyperplane $H:= \{t_4 = 0\}$. The following variety is thus natural to study.
\begin{definition}\label{def: Ebf}
    Fix $p_1, \dots, p_4 \in \PP^3$ in linearly general position. We let $\Ebf \subset \GG(2, |\Ocal_{\PP^3}(2)|)$ denote the locus of nets of quadrics $\Gamma \in \GG(2, |\Ocal_{\PP^3}(2)|)$ such that $\supp(\Bs\, \Gamma) = p_1 + \dots + p_4$, $\mult_{p_i}(\Bs\,\Gamma) = 2$ for $i = 1, \dots, 4$.
\end{definition}

We define the rational map $\rho$ to be the restriction map
\[
\rho : \GG(2, |\Ocal_{\PP^4}(2)|) \dashrightarrow \GG(2, |\Ocal_{\PP^3}(2)|), \quad \Lambda \mapsto \Lambda|_H.
\]
The exact structure of the fibres of $\rho$ is readily identifiable. Let $\Gamma \in \GG(2, |\Ocal_{\PP^3}(2)|)$ and notice that $\rho^{-1}(\Gamma) \approx \GG(2, \PP V_\Gamma)$ where $V_\Gamma$ is the vector space $V_\Gamma := \Gamma \oplus t_4\cdot \H^0(\PP^4, \Ocal_{\PP^4}(1))$. Set $\CC^5 := t_4\cdot \H^0(\PP^4, \Ocal_{\PP^4}(1))$ and let $\Scal \rarr \GG(2, |\Ocal_{\PP^4}(2)|)$ be the universal subbundle. We discover that there exists a birational isomorphism
\begin{align}\label{eq: key grassmann bd struct on Grassmannian of nets of quadrics in P4}
\GG(2, |\Ocal_{\PP^4}(2)|) \approx G(3, \Scal \oplus (\CC^5\times \GG(2, |\Ocal_{\PP^3}(2)|))
\end{align}
over $\GG(2, |\Ocal_{\PP^3}(2)|)$.

\vskip 0.5em
The exact locus of indeterminacy of $\rho$ is easy to describe. The restriction of a net $\Lambda \in \GG(2, |\Ocal_{\PP^3}(2)|)$ to $H$ is undefined as a net of quadrics if and only if there exists a quadric $Q \in \Lambda$ that contains $H$, meaning $\Lambda \cap t_4 \cdot \H^0(\PP^4, \Ocal_{\PP^4}(1)) \neq 0$. Therefore, the exact fibre of $\rho$ over $\Gamma$ is
\begin{align}\label{eq: fibre of rho is complement of Schubert cycle}
\rho^{-1}(\Gamma) = \{\Lambda \in G(3, V_\Gamma)\colon \Lambda \cap t_4\cdot \H^0(\PP^4, \Ocal_{\PP^4}(1)) = 0\}.
\end{align}
The following object is thus natural to investigate.

\begin{definition}\label{def: Fbf}
    We define the subvariety $\Fbf \subset \GG(2, \Ocal_{\PP^4}(2))$ to be the preimage of $\Ebf$ under the rational map $\rho$, inside the domain of $\rho$.
\end{definition}

Recall the fixed points $p_1, \dots, p_4 \in H \subset \PP^4$ in linearly general position, and let $\P G\subset \PGL(5)$ denote the stabiliser subgroup these points. In suitable coordinates,
\[
\P G \cong \left\{
\begin{bmatrix}
    * & 0 & 0 & 0 & *\\
    0 & * & 0 & 0 & * \\
    0 & 0 & * & 0 & * \\
    0 & 0 & 0 & * & * \\
    0 & 0 & 0 & 0 & *
\end{bmatrix}
\right\} \subset \PGL(5).
\]
Thus, we have arrived at the following birational model for $\Mcal_{5, 4}^{(2)}$.
\begin{proposition}\label{prop: birational model for differential stratum}
    The map $\Fbf \gitquot \P G \dashrightarrow \diffstrat$, $\Lambda \mapsto [\Bs(\Lambda), p_1, \dots, p_4]$, is a birational isomorphism onto the locus of non--hyperelliptic, non--trigonal curves.
\end{proposition}

Building on Proposition \ref{prop: birational model for differential stratum}, we now go on to tackle the problem of the rationality of the differential stratum $\diffstrat$.

\begin{proof}[Proof of Theorem \ref{thm: rationality of diff stratum}]
    We first claim that $\Fbf$ is a rational variety. As it follows from (\ref{eq: key grassmann bd struct on Grassmannian of nets of quadrics in P4}), the restriction $\rho|_\Fbf : \Fbf \rarr \Ebf$ is a Grassmann bundle so the rationality of $\Fbf$ reduces to the rationality of $\Ebf$. This latter space is rational for the following reason.

    \vskip 0.5em
    Let $\P G ' \subset \PGL(4) = \Aut(H)$ be the stabiliser group of the points $p_1, \dots, p_4$ inside the automorphism group of the hyperplane $H \cong \PP^3$. This group identifies with a diagonal subgroup in $\PGL(4)$, $\P G' \cong (\CC^*)^3$. Define $\Xscr$ to be the following incidence variety:
    \begin{align*}
    \begin{tikzcd}[ampersand replacement=\&]
    	\& {\Xscr := \{(P, \Gamma) \in \GG\bigl(1, \bigl|\Ical_{\{p_i\}_{i=1}^4}(2)\bigr|\bigl) \times \Ebf\colon P \subset \Gamma\}} \\
    	{\GG\bigl(1, \bigl|\Ical_{\{p_i\}_{i=1}^4}(2)\bigl|\bigr)} \&\& \Ebf,
    	\arrow["{\pi_1}"', from=1-2, to=2-1]
    	\arrow["{\pi_2}", from=1-2, to=2-3]
    \end{tikzcd}
    \end{align*}
    where $\pi_1 : \Xscr \rarr \GG(1, |\Ocal_{\PP^3}(2)|)$ and $\pi_2 : \Xscr \rarr \Ebf$ are the projections onto the first and second factors respectively. 

    \vskip 0.5em
    We analyse $\Xscr$ more closely. Let $P \in \GG(1, |\Ocal_{\PP^3}(2)|)$ be a general pencil of quadrics and set $E := \Bs \, P$ for the base locus of $P$. This is an elliptic curve of degree 4 in $\PP^3$. To give a net of quadrics $\Gamma \in \Ebf$ such that $P \subset \Gamma$ is equivalent to giving a quadric totally tangent to $E$. This shows that $\pi_1$ is an isomorphism onto the locus of pencils of quadrics $\Gamma$ for which, on the base locus $E := \Bs(\Gamma)$, the line bundle $\Ocal_E(1)(-p_1 - \dots - p_4)$ has order 2 in $\Pic^0(E)$. This analysis implies that there exists a birational isomorphism
    \begin{align}
    \Xscr \gitquot \P G' \approx \Rcal_{1;4}
    \end{align}
    where $\Rcal_{1;m}$ is the moduli space
    \[
    \Rcal_{1;m}:= \{[E, p_1, \dots, p_m, \eta]\colon [E, p_1, \dots, p_m] \in \Mcal_{1, m} \text{ and } \eta \in \Pic^0(E)[2]\text{ non--trivial}\}.
    \]
    The space $\Rcal_{1;4}$ is known to be (uni)rational so that, taking into account the induced projection $\Xscr\gitquot\P G' \rarr \Ebf\gitquot \P G'$, we deduce that $\Ebf \gitquot \P G'$ is unirational as well. 

    \vskip 0.5em
    We now point out that $\pi_2 : \Xscr \rarr \Ebf$ is a $\PP^2$--bundle, implying that
    \[
    \dim \Ebf\gitquot\P G' = 2,
    \]
    that is, $\Ebf / \P G'$ is a unirational surface. It is now a consequence of Castelnuovo's Rationality Criterion that $\Ebf \gitquot \P G'$ is in fact rational. Since $\Ebf \rarr \Ebf \gitquot \P G'$ is a principal $\P G'$--bundle in the \'etale topology and $\P G' \cong (\CC^*)^3$ is a simple group, this is a principal bundle in the Zariski topology. This proves the rationality of $\Ebf$, implying the rationality of $\Fbf$. By Myiata's Theorem from \cite{miyata1971InvariantsCertainGroups} on the rationality of triangular representations, the rationality of $\Fbf \gitquot \P G$ follows. Coupled with Proposition \ref{prop: birational model for differential stratum}, this finishes the proof of the rationality of $\Mcal_{5, 4}^{(2)}$.
\end{proof}

\section{Nets of quadrics in \texorpdfstring{$\PP^3$}{} and Segre--general subpencils}

Motivated by the birational model discussed in Section \ref{section: rationality of diffstrat} and for later use as well, the problem of this subsection is to exclude the possibility of obtaining by restriction nets of quadrics in $\PP^3$ with the following highly degenerate property: \textit{every subpencil of the net has singular base locus.} Let $[C, \eta] \in \Scal_5^-$ be a non--hyperelliptic non--trigonal odd spin curve of genus 5 with $H \subset \PP^3$ the hyperplane giving rise $\eta$ on $C$. Write $\eta = \Ocal_C(p_1 + \dots + p_4)$ for not necessarily distinct points $p_1, \dots, p_4\in C$. Let $\Lambda \in \GG(2, |\Ocal_{\PP^4}(2)|)$ be the net of quadrics whose complete intersections is $C$ and $\Lambda_H:= \Lambda|_H$ be the restriction to $H$.

\vskip 0.5em
Segre's classification of pencils of quadrics tells us in particular that if $P$ has Segre symbol $[1, 1, 1, 1]$ (meaning that the discriminant polynomial $\disc\, P$ has no multiple roots), the base locus $\Bs\, P$ is smooth. Therefore, we must discard two possibilities according to the discriminant $\Delta := \disc\, \Lambda_H$:
\begin{enumerate}[label=$\roman*)$]
    \item $\Delta = |\Lambda_H|$, i.e. every quadric in $\Lambda_H$ is singular. As it turns out, this is equivalent to the net $\Lambda_H$ failing to be semistable (\cite[p. 233]{wall1978netsQuadricsTheta}).
    \item $\Delta$ has multiple components.
\end{enumerate}
In discarding these two cases, the following lemma will shortly prove its usefulness.
\begin{lemma}\label{lemma: uniqueness of quadcs singular at given point in Gamma net of quadrics in P3}
    No two quadrics in $\Lambda_H$ share a common singularity amongst the points $p_i$.
\end{lemma}
\begin{proof}
    Suppose there exist quadrics $q_i, q_i' \in \Lambda_H$ linearly independent both singular at $p_i$, and let $P := \langle q_i, q_i' \rangle \subset \Lambda_H$ be the pencil they generate. Let $\widetilde P \subset \Lambda$ be any lift of the pencil $P$, and consider the vector space
    \[
    V_{p_i}:= \{\ell(p_i) = 0\} \subset \Lambda_H \oplus t_4 \cdot \H^0(\PP^4, \Ocal_{\PP^4}(1)).
    \]
    This is the hyperplane of all quadrics in $Q \in \Lambda_H \oplus t_4 \cdot \H^0(\PP^4, \Ocal_{\PP^4}(1))$ which, if written $Q = q + t_4 \cdot \ell$ for $q \in \Lambda_H$ and $\ell \in \H^0(\PP^4, \Ocal_{\PP^4}(1))$, $\ell(p_i) = 0$ holds. Let $Q_0 \in \widetilde P \cap V_{p_i}$ be non--zero (which exists for dimension reasons). One checks easily that $Q_0$ is singular at $p_i$ so that we have found a quadric $Q_0 \in \Lambda$ which is singular at $p_i \in \Bs\, \Lambda$. This implies that $C = \Bs\, \Lambda$ is singular at $p_i$, a contradiction.
\end{proof}

We now turn our attention back to cases $i)$ and $ii)$. Possibility $i)$ is easy to discard if one recalls Wall's classification of non--semistable nets of quadrics in $\PP^3$ as shown in (\ref{eq: non--semistable nets of quadrics in P3}).
\begin{align}\label{eq: non--semistable nets of quadrics in P3}
\begin{pmatrix}
    0 & 0 & 0 & 0 \\
    0 & * & * & * \\
    0 & * & * & * \\
    0 & * & * & *
\end{pmatrix} \quad \text{or} \quad
\begin{pmatrix}
    0 & 0 & 0 & * \\
    0 & 0 & 0 & * \\
    0 & 0 & * & * \\
    * & * & * & *
\end{pmatrix}.
\end{align}
In the second matrix, the base locus contains a line and this cannot happen since $\dim \Bs(\Lambda_H) = 0$. As for the first matrix, $[1:0:0:0] \in \Bs\, \Lambda_H$ is a singularity shared by all of the quadrics in the net, contradicting Lemma \ref{lemma: existence of pencils of quadrics with non--singular locus}.

\vskip 0.5em
We can now assume that $\Delta \neq |\Lambda_H|$, and we want to analyse what happens when $\Delta$ is non--reduced. The following lemma follows from \cite[Theorem 1.6]{wall1978netsQuadricsTheta}
\begin{lemma}
    Let $\Gamma \in \GG(2, |\Ocal_{\PP^3}(2)|)$ be a net of quadrics. If $\Gamma$ is stable, the discriminant $\Delta$ has a multiple component if and only if $\dim (\Bs\,\Gamma) > 0$.
\end{lemma}

\vskip 0.5em
We are thus left to deal with unstable nets of quadrics in $\PP^3$ which are semistable. Namely, we have the following possibilities (cf. \cite[Theorem 0.1]{wall1978netsQuadricsTheta}) for all matrices in the given net:
\[
\begin{pmatrix}
    0 & 0 & 0 & * \\
    0 & * & * & * \\
    0 & * & * & * \\
    * & * & * & *
\end{pmatrix} \quad \text{or} \quad
\begin{pmatrix}
    0 & 0 & * & * \\
    0 & 0 & * & * \\
    * & * & * & * \\
    * & * & * & *
\end{pmatrix}.
\]
Notice again that the second type of matrix is not possible by dimension reasons as the base locus of such a net contains a line.

\vskip 0.5em
We now take a closer look at the first type. From Wall's classification \cite[1.6]{wall1978netsQuadricsTheta} of semistable nonstable quadrics, only cases (i) and (ii) are worth considering as otherwise the base locus contains a component of dimension $>0$. Take for example (i). This is the scenario where all quadrics are tangent to a fixed plane $\pi \subset \PP^3$ at a point $x_1 \in \pi$. Let now $P$ be any pencil in such a net. Since $x_1 \in \Bs\, P$ is a singularity of $\Bs\, P$ (the quadrics share a common tangent plane there), there exists a quadric $q \in P$ singular at $x_1$. Picking a different pencil $P'$ in the net which avoids $q$, we find another quadric $q' \in P'$ which is also singular at $x_1$, contradicting Lemma \ref{lemma: existence of pencils of quadrics with non--singular locus}. Case (ii) is dealt with by a similar token.

\vskip 0.5em
The discussion in this subsection has the following useful consequence.
\begin{lemma}\label{lemma: existence of pencils of quadrics with non--singular locus}
    Let $[C, \eta] \in \sfive$, $\eta = \Ocal_C(D)$, be non--hyperelliptic, non--trigonal and view $C$ canonically embedded in $\PP^4$. Let $H \in (\PP^4)^\vee$ be such that $C \cdot H = 2D$. The discriminant $\Delta := \disc \, \Lambda|_H$ is not the entire net $|\Lambda|$, nor does it have multiple components. In particular, the general pencil $P \subset |\Lambda|$ has base locus a smooth elliptic curve.
\end{lemma}

\section{The Brill--Noether general locus}\label{section: BN general locus}

\begin{definition}
    Let $\sfivetilde \subset \sfive$ denote the locus of non--hyperelliptic, non--trigonal curves with reduced theta characteristic.
\end{definition}

The main object of interest in this section is the open subset $\widetilde \Scal_5^-$ in $\Scal_5^-$. Using a variant of the model described in Section \ref{section: rationality of diffstrat}, we prove the following proposition.
\begin{proposition}\label{prop: chow ring of sfivetilde is tautological}
    The Chow and cohomology rings of $\sfivetilde$ is generated by tautological classes.
\end{proposition}

Before going into the proof, we define some of the varieties that will be necessary along the way. Consider the following incidence variety:
\begin{align}
    \Fscr := \left\{(\Lambda, H) \in \GG(2, |\Ocal_{\PP^4}(2)|) \times (\PP^4)^\vee \colon \begin{matrix}
    \supp(\Bs\, \Lambda|_H) = p_1 + \dots + p_4, \text{ $p_i \in H$ are} \\
    \text{distinct, } \mult_{p_i}(\Bs\, \Lambda|_H) = 2 \text{ for }i = 1, \dots, 4
    \end{matrix}\right\}.
\end{align}
This is just the version of the variety $\Fbf$ in Definition \ref{def: Fbf} where we allow both the points $p_i$ in the support of $\Bs(\Lambda|_H)$ as well as the hyperplane $H$ to vary. Consider the two projections $\pr_1 : \Fscr \rarr \GG(2, |\Ocal_{\PP^4}(2)|)$ and $\pr_2 : \Fscr \rarr (\PP^4)^\vee$ onto the first and second factors. Notice that $\pr_1$ is finite as the fibre $\pr_1^{-1}(\Lambda)$ parametrises totally tangent hyperplanes to $C := \Bs(\Lambda)$. It follows that
\[
\dim \Fscr = \dim \GG(2, |\Ocal_{\PP^4}(2)|) = \dim G(3, 15) = 36.
\]
The group $\PGL(5)$ acts on $\Fscr \subset \GG(2, |\Ocal_{\PP^4}(2)|) \times (\PP^4)^\vee$ and we find the following birational model for $\sfive$. Let $\mu : \Fscr/ \PGL(5) \dashrightarrow \sfive$ be the map
\[
[\Lambda, H] \mapsto [C, \eta] \quad \textup{for }C = \Bs\, \Lambda\, , \eta = \Ocal_C(p_1 + \dots + p_4)
\]
where $C \cdot H = 2p_1 + \dots + 2p_4$. The following result now holds:
\begin{proposition}
    The map $\mu : \Fscr / \PGL(5) \dashrightarrow \sfive$ is a a birational isomorphism onto the locus $\widetilde\Scal_5^-$.
\end{proposition}

\begin{proof}[Proof of Proposition \ref{prop: chow ring of sfivetilde is tautological}]
We start by observing that $\Fscr/ \SL(5) \rarr \Fscr/\PGL(5)$ is a $\mu_5$--banded gerbe so that, since we are working with rational coefficients, we obtain an isomorphism of Chow rings $\CH^\bullet(\Fscr/\PGL(5)) \cong \CH^\bullet(\Fscr/\SL(5))$. Let $\Vcal \rarr \Fscr/\SL(5)$ be the vector bundle induced by $\Fscr \rarr \Fscr/\SL(5)$ so that $c_1(\Vcal) = 0$. By Theorem \ref{thm: vistoli thm chow groups of quotients}, the Chow ring $\CH^\bullet(\Fscr/\SL(5))$ is generated as a $\QQ$--algebra by the Chern classes $\{c_i(\Vcal)\}_{i = 2}^5$ and any lift of generators for $\CH^\bullet(\Fscr)$. The same type of argument as in \cite[3.10]{vakil2015chowRingM6} gives that $c_i(\Vcal)$ are all polynomials in the $\lambda$--classes. We next want to focus on $\CH^\bullet(\Fscr)$.

\vskip 0.5em
The key auxiliary variety needed for our current goal is the following analogue of the variety $\Ebf$ from Section \ref{section: rationality of diffstrat}. Let $\Wcal \rarr (\PP^4)^\vee$ be the universal bundle so that $\PP \Wcal \rarr (\PP^4)^\vee$ is the universal hyperplane, and let $\pi: G(3,\Sym^2(\Wcal^\vee)) \rarr (\PP^4)^\vee$ be the Grassmann bundle of 3--planes in $\Sym^2(\Wcal^\vee)$. Define
\begin{align}\label{eq: definition of Bscr}
    \Bscr := \left\{\Gamma \in G(3, \Sym^2(\Wcal^\vee))  \colon \begin{matrix}
        \text{if $H := \pi(\Gamma)$, $\supp\, \Bs(\Gamma) = p_1 + \dots + p_4 $,} \\
        \text{then }\mult_{p_i}(\Bs\, \Gamma) = 2 \text{ for }i = 1, \dots, 4
    \end{matrix}\right\}.
\end{align}
As a subvariety of $G(3, \Sym^2(\Wcal^\vee))$, $\Bscr$ admits a projection $\Bscr \rarr (\PP^4)^\vee$. Define $\varrho$ to be the restriction map $\varrho : \Yscr \rarr \Bscr$, given by
\[
\varrho : (\Lambda, H) \mapsto \Lambda|_H \in \GG(2, |\Ocal_H(2)|).
\]
As in Section \ref{section: rationality of diffstrat} where the hyperplane $H$ and the points $p_1, \dots, p_4$ were fixed, we see that $\varrho : \Fscr \rarr \Bscr$ is a Zariski-open subset of a Grassmann bundle over $\Bscr$. Thus, just as in (\ref{eq: key grassmann bd struct on Grassmannian of nets of quadrics in P4}) and (\ref{eq: fibre of rho is complement of Schubert cycle}), we find that $\Fscr$ is the complement of a Schubert cycle of complementary type, which is a linear section in the Plücker embedding of the Grassmann bundle containing $\Yscr$. The Andreotti--Frankel Vanishing Theorem for the cohomology of affine varieties (which holds for Chow rings of Grassmann bundles) implies that we have an isomorphism $\CH^\bullet(\Bscr) \cong \CH^\bullet(\Fscr)$.

\vskip 0.5em
We now shift our attention to $\Bscr \subset G(3, \Sym^2(\Wcal^\vee)) \rarr (\PP^4)^\vee$. Consider the incidence variety $\Ascr$ given by
\begin{align}\label{eq: key incidence variety Ascr}
\Ascr := \{(P, \Gamma) \in G(2, \Sym^2(\Wcal^\vee)) \times_{(\PP^4)^\vee}\Bscr \colon P \subset \Gamma\} 
\end{align}
and the two projections
\[\begin{tikzcd}
	& \Ascr \\
	{G(2, \Sym^2(\Wcal^\vee))} && {\Bscr}
	\arrow["{q_1}"', from=1-2, to=2-1]
	\arrow["{q_2}", from=1-2, to=2-3]
\end{tikzcd}.\]
The projection $q_2 : \Ascr \rarr \Bscr$ is a $\PP^2$--bundle with fibre $q_2^{-1}(\Gamma) = G(2, \Gamma)$ which by the Projective Bundle Formula gives that $\CH^\bullet(\Bscr) = \CH^\bullet(\Ascr) / \langle c_1(\Ocal_\Ascr(1))\rangle$. The Chow ring of $\Ascr$ is understood by relating $\Ascr$ to $\Rcal_{1;[4]}$. Namely, just as in Subsection \ref{section: rationality of diffstrat}, we find the birational isomorphism $\Ascr/\PGL(5) \xrightarrow[]{\approx} \Rcal_{1;[4]}$. The Chow ring of $\Rcal_{1;4}$ is known to be generated by the $\psi$--classes by \cite{krug2012thesisSpinPrym}, giving that the Chow ring of $\Rcal_{1;[4]}$ is generated by symmetric polynomials in them. In particular, the Chow ring of $\Ascr/\PGL(5) \approx \Rcal_{1;4}$ is generated by symmetric polynomials in (pullbacks of) the $\psi$--classes. Vistoli's Theorem \ref{thm: vistoli thm chow groups of quotients} now gives that $\CH^\bullet(\Ascr)$ is generated by the $\psi$--classes as well.

\vskip 0.5em
Tracing back through the arguments so far we find that $\CH^\bullet(\widetilde\Scal_5^-)$ is generated by symmetric polynomials in the $\psi$--classes coming from the identification $\Mcal_{5, 4}^{(2)}/\Sfrak_4 \cong \sfive$. Witney's Formula together with the relation $\eta^{\otimes 2} = \omega$ on the universal curve over $\sfive$ finally allow us to express the $\psi$--classes in terms of the $\lambda$--classes. This finishes the proof of the Proposition.
\end{proof}

\subsection{Conclusion on all differential strata in $\sfive$}

Let $\mu = (2\mu_1, \dots, 2\mu_n)$ be a partition of $8$ where $\mu_i > 0$ for all $i$. We denote by $\widetilde\Scal_5^-(\mu) \subset\sfive$ the image of the non--hyperelliptic, non--trigonal locus in the differential stratum 
\[
\Mcal_{5, n}^\mu := \{[C, p_1, \dots, p_n] \in \Mcal_{5, n}\colon \Ocal_C(2\mu_1p_1 + \dots + 2\mu_np_n) \simeq \omega_C\}.
\]
In this subsection, we explain how almost identical arguments to the proof of Proposition \ref{prop: chow ring of sfivetilde is tautological} yield that $\CH^\bullet\bigl(\widetilde\Scal_5^-(\mu)\bigr)$ yield the following theorem.
\begin{proposition}\label{prop: tautology sfive tile mu}
    The Chow and cohomology rings of the stratum $\sfivetilde(\mu)$ are generated by restrictions of tautological classes for any partition $\mu = (2\mu_1, \dots, 2\mu_n)$ of 8 such that $\mu_i > 0$ for all $i$. Moreover, all the classes $\bigl[\widetilde{\Scal}_5^-(\mu)\bigr]$ are tautological inside the Chow ring of the non--hyprelliptic, non--trigonal locus $\widetilde\Scal^-_5 \setminus (\Scal\Hcal_5^- \cup \Scal\Tcal_5^-)$.
\end{proposition}
\begin{proof}
We start by addressing the first claim. Consider the following generalisation of (\ref{eq: definition of Bscr}) for all even partitions $\mu = (2\mu_1, \dots, 2\mu_n)$ of 8:
\begin{align*}
    \Bscr(\mu) := \left\{\Gamma \in G(3, \Sym^2(\Wcal^\vee))  \colon \begin{matrix}
        \text{if $H := \pi(\Gamma)$, }\supp(\Bs\, \Gamma) = p_1+\dots+ p_n, \\
        \text{then }\mult_{p_i}(\Bs\, \Gamma) = \mu_i\text{ for } i = 1, \dots, n
    \end{matrix}\right\}.
\end{align*}
Define the incidence variety $\Ascr(\mu)$ by
\begin{align*}
\Ascr(\mu) := \{(P, \Gamma) \in G(2, \Sym^2(\Wcal^\vee)) \times_{(\PP^4)^\vee}\Bscr(\mu) \colon P \subset \Gamma\} 
\end{align*}
By Lemma \ref{lemma: existence of pencils of quadrics with non--singular locus}, the open subset
\[
\Ucal(\mu) := \{(P, \Gamma) \in \Ascr(\mu)\colon \Bs\, P\text{ is smooth}\} \subset \Ascr(\mu)
\]
surjects onto $\Bscr(\mu)$. We are thus reduced to showing that the Chow ring of $\Ucal$ is generated by tautological classes. This is done almost identically to Subsection \ref{section: BN general locus}, with the following differences:
\vskip 0.5em
\begin{itemize}[label=--]
    \item $\Ucal(2, 2, 4)\gitquot \PGL(5) \cong \Rcal_{1;1+[2]} := \Rcal_{1;4}/\Sfrak_2$ where $\Sfrak_2$ permutes the last two marked points.
    \vskip 0.5em
    \item $\Ucal(4, 4)\gitquot \PGL(5) \cong \Rcal_{1;2}$.
    \vskip 0.5em
    \item $\Ucal(8) \gitquot \PGL(5) \cong \Rcal_{1;1}$.
\end{itemize}
As in our study of the locus of spin curves with reduced theta characteristics, we see that all the loci $\widetilde\Scal_5^-(\mu)$ have tautological Chow ring by results in \cite{krug2012thesisSpinPrym}.

\vskip 0.5em
We now claim that the classes $[\widetilde\Scal_5^-(\mu)]$ are tautological. Given $\mu = (2\mu_1, \dots, 2\mu_n)$ a partition of $2g - 2$, over the stack of pointed spin curves $\Scal_{g;n}^-:= \sg^- \times_{\mg}\mgn$, consider the evaluation map
\[
\ev(C, p_1, \dots, p_n, \eta) : \H^0(C, \eta) \rarr \H^0(C, \eta|_{\mu_1p_1 + \dots + \mu_np_n}).
\]
The dimension of the degeneracy locus $D(\ev)$ is
\[
\dim D(\ev) = 3g - 3 + n - \sum_{i = 1}^n\mu_i,
\]
which matches the expected one. By Porteus' Formula, over the locus $\{h^0(C, \eta) = 1\}$ which in genus 5 is the same as the non-hyperelliptic locus, the tautology of the classes $[\widetilde\Scal_5^-(\mu)]$ follows.
\end{proof}

Using push and pull, the results of this section are summarised in the following corollary.
\begin{corollary}\label{cor: CH = R for nonhyp nontrig}
    The Chow ring $\CH^\bullet(\sfive \setminus(\Scal\Tcal_5^- \cup \Scal\Hcal_5^-))$ of the locus of non--hyprelliptic, non--trigonal spin curves of genus 5 is generated by tautological classes.
\end{corollary}

\section{The trigonal and hyperelliptic loci}

By Corollary \ref{cor: CH = R for nonhyp nontrig}, we are left to analyse the hyperelliptic and trigonal loci.
\begin{definition}
    \hspace{-50pt}
    \begin{enumerate}[label=$\roman*)$]
        \item We denote by $\Hcal_g \subset \mg$ the hyperelliptic locus in $\mg$ and by $\Scal\Hcal_g \subset \sg$ the trigonal locus to $\sg$.
        \item We denote by $\Tcal_g \subset \mg$ the trigonal locus in $\mg$ and by $\Scal\Tcal_g^\pm \subset \sg^\pm$ the pullback of the trigonal locus to $\sg^\pm$.
    \end{enumerate}
\end{definition}

\subsection{Hyperelliptic spin curves}
We begin by recalling the following well--known description of theta characteristics on a hyperelliptic curve (cf. \cite[p.288, Exercises 32 and 33]{ACGH}). 
\begin{proposition}\label{prop: theta characteristics on hyperelliptic curves}
	Let $\alpha : C \rarr \PP^1$ be a smooth hyperelliptic curve of genus $g$ with hyperelliptic linear system $|D| = |\alpha^*\Ocal_{\PP^1}(1)|$. Let $x_1 + \dots + x_{2g + 2}$ be the ramification divisor of $\alpha$.
	\begin{enumerate}[label=\roman*)]
		\item Any theta characteristic on $C$ is of the form $\eta = \Ocal_C(E)$ where
		\[
		E = kD + x_{i_1} + \dots + x_{i_{g - 1 - 2k}}
		\]
		for some $-1 \leq k \leq \left\lfloor\frac{g - 1}{2}\right\rfloor$.
		\item The expression of $\eta$ in i) is unique if $k \geq 0$. If $k = -1$, then
		\[
		-D + x_{i_1} + \dots + x_{i_{g + 1}} \sim -D + x_{j_1} + \dots + x_{j_{g + 1}}
		\]
		if and only if $\{{i_1}, \dots, {i_{g+1}}\} \cup \{{j_1}, \dots, {j_{g + 1}}\} = \{1, \dots, 2g + 2\}$.
		\item If a theta characteristic $\eta$ has the form as in i), then $h^0(C, \eta) = k + 1$.
	\end{enumerate}
\end{proposition}

Proposition \ref{prop: theta characteristics on hyperelliptic curves} has the following immediate consequence. Consider the loci
\begin{align}
	\Scal\Hcal_g^k := \{[C, \eta] \in \Scal\Hcal_g\colon h^0(C, \eta) = k + 1\}
\end{align}
for $-1 \le k \le \lfloor\frac{g -1}{2}\rfloor$. The following corollary is now clear.
\begin{corollary}
	The irreducible components of the hyperelliptic locus $\Scal\Hcal_g$ are the subvarieties $\Hcal{\Scal}_g^k$ for $-1 \leq k \leq \lfloor\frac{g - 1}{2}\rfloor$.
\end{corollary}

We construct, for each $-1 \leq k \leq \lfloor\frac{g -1}{2}\rfloor$, a finite étale map $a_k :\Mcal_{0, 2g + 2} \rarr \Scal\Hcal_g^k$ which sends $[\PP^1, y_1, \dots, y_{2g + 2}] \mapsto [C, \eta]$, where $\alpha : C \rarr \PP^1$ is the unique double cover of $\PP^1$ branched at $y_1 + \dots + y_{2g + 2}$ and, setting $\{x_i\} = \alpha^{-1}(y_i)$, $\eta = \Ocal_C(kD + x_1 + \dots + x_{g-1 - 2k})$. These maps induce isomorphisms
\begin{align}\label{eq: concrete description SH_g(k)}
\Scal\Hcal_g^k \cong \Mcal_{0, 2g + 2} / \Sfrak_{g - 1 - 2k} \rtimes \Sfrak_{g + 3 +2k}.
\end{align}
From (\ref{eq: concrete description SH_g(k)}) we deduce that the $\Scal\Hcal_g^k$ is Chow--free.

\vskip 0.5em
With all this background in mind, we can now tackle the main theorem of this subsection. The following holds:
\begin{theorem}\label{thm: tautology of chow of moduli hyperelliptic spin}
    The loci $\Scal\Hcal_g^k$ are all Chow--free. Moreover, the classes $[\Scal\Hcal_{5}^k] \in \CH^\bullet(\Scal_{5}^-)$ are tautological for $k = 0, 2$.
\end{theorem}
\begin{proof}
    Only the tautology of the classes $[\Scal\Hcal_5^k]$ for $k = 0, 2$ remains to be addressed. Since the hyperelliptic locus in $\mg$ is tautological, it suffices to show $[\Scal\Hcal_5^k] \in \R^\bullet(\Scal_5^-)$. Notice that $\Scal\Hcal_5^2$ can be written as the locus
    \[
    \Scal\Hcal_5^2 = \{[C, \eta]\in \Scal_5^-\colon h^0(C, \eta) \ge 3\}.
    \]
    This identification follows by Clifford's Theorem and Proposition \ref{prop: theta characteristics on hyperelliptic curves}. Let $d \gg 0$ and $p \in C$. Consider the cohomology sequence
    \[
    0 \rarr \H^0(C, \eta) \rarr \H^0(C, \eta(dp)) \xrightarrow[]{\ev(C, p, \eta)} \H^0(C, \eta(dp)/\eta) \rarr \H^1(C, \eta) \rarr 0.
    \]
    This picture of course globalises, and we are interested in loci of the form
    \[
    \{[C, p, \eta]\colon\dim \Ker\left(\ev(C, p, \eta)\right) \ge k + 1\}.
    \]
    We now express the preimage of $\Scal\Hcal_5^2$ in $\Scal_{5;1}^- := \sfive \times_{\Mcal_5} \Mcal_{5, 1}$ as a symmetric degeneracy locus. Let $(f : \Ccal_{g;1}^- \rarr \Scal_{g;1}^-, \sigma : \Scal_{g;1}^- \rarr \Ccal_{g,1}^-)$ be the universal curve over the stack $\Scal_{g;1}^- = \Scal_g^- \times_{\mg} \Mcal_{g;1}$. Let $\eta$ be the universal theta characteristic on $\Ccal_{g;1}^-$. Consider the long exact higher direct image sequence
    \[
    0 \rarr f_*\eta \rarr f_*\eta(d\sigma) \xrightarrow[]{\ev} f_*\left[\eta(d\sigma) / \eta\right] \rarr \R^1f_*\eta \rarr 0.
    \]
    Set $\Acal:= f_*\eta(d\sigma)$ and $\Bcal:= f_*(\eta \otimes \Ocal_{d\sigma}(d\sigma))$. The subvariety
    \[
    \{[C, p, \eta]\colon h^0(C, \eta) \ge k + 1\} \subset \Scal_{g;1}^-
    \]
    for $k$ even is the degeneracy locus of $\ev$. We claim that $\Bcal \cong \Acal^\vee$. Indeed, we have that, by Serre duality,
    \[
    \Bcal^\vee = (f_*(\eta(d\sigma)/\eta))^\vee \cong \R^1f_*(\omega_f \otimes (\eta\otimes \Ocal_{d\sigma}(d\sigma))^\vee) = \R^1f_*(\eta \otimes \Ocal_{d\sigma}(-d\sigma)).
    \]
    We examine closer the rank $d$ vector bundle $\R^1f_*(\eta \otimes \Ocal_{d\sigma}(-d\sigma))$. Consider the short exact sequence
    \[
    0 \rarr \eta(-2d\sigma) \rarr \eta(-d\sigma) \rarr \eta \otimes \Ocal_{d \sigma}(-d\sigma) \rarr 0,
    \]
    giving rise to the map of vector bundles
    \[
    \R^1f_*\eta(-d\sigma) \rarr \R^1f_*(\eta \otimes \Ocal_{d\sigma}(-d\sigma)) \rarr 0.
    \]
    Notice that the fibres of $\R^1f_*\eta(-d\sigma)$ are
    \[
    (\R^1f_*\eta(-d\sigma))_{[C, p]} = \H^1(C, \eta_C(-dp)),
    \]
    and we compute $h^1(C, \eta_C(-dp)) = h^0(C, \eta_C(dp)) = d$. Therefore, the homomorphism $\R^1f_*\eta(-d\sigma) \rarr \R^1f_*(\eta\otimes \Ocal_{d\sigma}(-d\sigma))$ is a surjection between vector bundles of the same rank, hence an isomorphism. Thus, we conclude that $\Acal \cong \Bcal^\vee$.

    \vskip 0.5em
    We can now apply \cite[Proposition 4]{pragaczDetSkewMatrices} if $\{h^0(C, \eta) \ge k + 1\}$ has expected codimension $\binom{k + 1}{2}$. In our situation, this locus is precisely $\Scal\Hcal_5^2$ which has codimension 3. It follows that the class $[\Scal\Hcal_5^2]$ is tautological.
\end{proof}

\subsection{The trigonal locus}

The following is the main result of thie subsection.
\begin{proposition}\label{prop: tautology chow TS5}
    The locus $\Tcal\Scal_5^-$ has Chow ring generated by restrictions of tautological classes from $\Scal_5^-$. Moreover, the class $[\Tcal\Scal_g^-]$ is tautological for any $g$.
\end{proposition}

Before going into the proof, we recall the well known interpretation of genus 5 trigonal curves as plane quintics with one singularity which is either a node or a cusp. More precisely, let $C$ be a non--hyperelliptic trigonal curve of genus 5 and $D$ a trigonal divisor. The linear system $|K_C - D|$ gives rise to a birational map $\varphi : C \rarr \PP^2$ such that $\Gamma := \Im \,\varphi$ is a plane quintic with one singularity of $\delta$--invariant 1 by the genus formula. If $\{p_0\} =\Sing(\Gamma)$ and $\varphi^{-1}(p_0) = \{p_0^+,p_0^-\}$, the points $p_0^+$, $p_0^-$ are distinct if $p_0$ is a node, and $p_0^+ = p_0^-$ if $p_0$ is a cusp. Notice that the divisor $p_0^+ + p_0^-$ is the unique divisor in $|K_C -2D|$.

\begin{proof}[Proof of Theorem \ref{prop: tautology chow TS5}]
To prove that $\CH^\bullet(\Tcal\Scal_5^-)$ generated by restrictions of tautological classes from $\Scal_5^-$, we split the trigonal locus into the following locally closed strata:
\begin{itemize}[label=--]
    \item $\Tcal^{(1)}$: The locus of pairs $[C, \eta]$ such that $K_C - 2D \sim 2p_0^+$ for some $p_0^+ \in C$ and $h^0(C, \eta(-p_0^+)) \neq 0$.
    \item $\Tcal^{(2)}$: The complement of $\Tcal^{(1)}$ in $\Tcal\Scal_5^-$.
\end{itemize}

\vskip 0.5em
We study first $\Tcal^{(2)}$ which is a dense open subset of $\Tcal\Scal_5^-$. Let $[C, \eta] \in \Tcal^{(2)}$ and maintain the notation in the first paragraph of this subsection with the convention that $p_0^+ = p_0^-$ if $\Gamma$ has a cusp at $p_0$. For parity reasons and also taking into account that $p_0^+ \notin \supp\, \eta$ it follows that $\eta$ is given by a smooth conic through $p_0$ which is tangent to $\Gamma$ at four other points.

\vskip 0.5em
Consider thus the following incidence variety:
\begin{align*}
    \Qscr^{(2)} := \{(\Gamma, p_0, \dots, p_4)\colon \Gamma \in |\Ocal_{\PP^2}(5)| \text{ is singular at $p_0$, }\Gamma\cdot C_\pbf = 2p_0 + \dots + 2p_4\}.
\end{align*}
In this context, if $\pbf = (p_0, \dots, p_4)$ then $C_\pbf$ stands for the unique conic through the five points $p_0, \dots, p_4$. Let $\pi : \Qscr^{(2)} \rarr (\PP^2)^5$ be the projection remembering the points $(p_0, \dots, p_4)$ and consider the open locus $\Ucal \subset (\PP^2)^5$ of tuples $\pbf$ such that $C_\pbf$ is smooth. We obtain a birational isomorphism
\begin{align*}
    \Tcal^{(2)} \approx (\Qscr^{(2)}\gitquot\PGL(3))/\Sfrak_4.
\end{align*}
This exhibits $\Tcal^{(2)}$ as an open subset of  $(\Qscr^{(2)}\gitquot\PGL(3))/\Sfrak_4$. Methods very similar to the ones described for the loci $\sfivetilde$ allow us to see that $\CH^\bullet(\Tcal^{(2)})$ has Chow ring generated by restrictions of tautological classes from $\Scal_5^-$.

\vskip 0.5em
We next continue to study $\Tcal^{(1)}$. In this case, $\eta = \Ocal_C(p_0^+ + p_1 + p_2 + p_3)$ where $p_1 + p_2 + p_3$ is a divisor on $C$ such that $2p_1 + 2p_2 + 2p_3 \sim 2D$. Indeed, we check that $2p_2 + 2p_2 + 2p_3 \sim K_C - 2p_0^+ = K_C - (K_C - 2D) = 2D$. If $2\cdot \ell_0$ is the tangent cone to $\Gamma$ at $p_0$, the divisor $2p_0^+ + 2p_1 + \dots + 2p_3$ is cut on $\Gamma$ by a conic tangent to $\ell_0$ to $p_0^+$ as well as to $\Gamma$ at the points $p_1, p_2, p_3$. This time one considers the incidence variety
\[
\Qscr^{(1)}:= \left\{(\Gamma, C, p_0, \dots, p_3)\colon \begin{matrix}
    C \in |\Ocal_{\PP^2}(5)| \text{ has a cusp at $p_0$, }C \in |\Ocal_{\PP^2}(2)| \\
     \text{such that }\Gamma \cdot C = 4p_0 + 2p_1 + 2p_2 + 2p_3
\end{matrix}\right\}.
\]
The existence of a birational isomorphism $\Tcal^{(1)} \approx (\Qscr^{(1)}\gitquot\PGL(3))/\Sfrak_3$ describing $\Tcal^{(1)}$ as an open subset of $(\Qscr^{(1)}\gitquot\PGL(3))/\Sfrak_3$, giving access to the Chow ring of $\Tcal^{(1)}$. By similar means as described for $\sfivetilde$, the Chow ring $\CH^\bullet(\Tcal^{(1)})$ is generated by restrictions of tautological classes from $\sfive$. The tautology of $[\Tcal^{(1)}] \in \CH^\bullet(\Tcal\Scal_5^-)$ follows by Porteus' formula since $\Tcal^{(1)}$ is given by $h^0(C, \eta(-p_0^+)) \ne 0$.
\end{proof}

We can now conclude the proof of the main result of this work.
\begin{proof}[Proof of Theorem \ref{thm: sfive has CH = R}]
The main theorem follows by assembling together Corollary \ref{cor: CH = R for nonhyp nontrig}, Theorem \ref{thm: tautology of chow of moduli hyperelliptic spin}, and Proposition \ref{prop: tautology chow TS5} which, by Excision and push and pull, give the tautology of $\CH^\bullet(\sfive)$.  
\end{proof}

\bibliography{bibliography}
\bibliographystyle{amsalpha}

\end{document}